%% file: modulesSLn.tex
\documentclass{amsart}

\usepackage{amsmath}
\usepackage{amssymb}
\usepackage{MnSymbol}
\usepackage{epsfig}  		
\usepackage{xypic}
\usepackage{bbm}
\usepackage{hyperref}
\usepackage{eucal}
\usepackage{mathrsfs}
\usepackage{cancel}
\usepackage{stmaryrd}
\usepackage{xcolor}
\usepackage{lineno}

\newtheorem*{bthm}{Theorem}

\newtheorem{mthm}{Theorem}
\renewcommand\themthm{\Alph{mthm}}

\newtheorem{thm}{Theorem}[section]
\newtheorem{prop}[thm]{Proposition}
\newtheorem{lem}[thm]{Lemma}
\newtheorem{cor}[thm]{Corollary}

\theoremstyle{definition}

\newtheorem{example}[thm]{Example}

\theoremstyle{remark}

\newtheorem{remark}[thm]{Remark}

\numberwithin{equation}{section}

\newcommand{\dietrich}[1]{}

\newcommand{\CC}{\mathbbm{C}}

\newcommand{\ZZ}{\mathbbm{Z}}

\newcommand{\Sym}{\mathrm{Sym}}

\newcommand{\group}{\mathrm}

\newcommand{\GL}{\group{GL}}
\newcommand{\SL}{\group{SL}}
\newcommand{\Sp}{\group{Sp}}
\newcommand{\SO}{\group{SO}}

\newcommand{\Spin}{\group{Spin}}

\newcommand{\Ch}{\mathsf{X}}

\newcommand{\mat}{\group{Mat}}

\renewcommand{\d}{\mathrm{d}}

\DeclareMathOperator{\grad}{\mathrm{grad}}

\newcommand{\frg}{\mathfrak{g}}
\newcommand{\fra}{\mathfrak{a}}

\newcommand{\frs}{\mathfrak{s}}

\newcommand{\ssl}{\mathfrak{sl}}
\newcommand{\ggl}{\mathfrak{gl}}

\newcommand{\Lie}{\mathfrak{Lie}}

\renewcommand{\rho}{\varrho}
\renewcommand{\phi}{\varphi}

\begin{document}

%
%

\title[\'Etale representations for
reductive algebraic groups]{\'Etale representations for
reductive algebraic groups with one-dimensional center}

\author[Burde]{Dietrich Burde}

\address{Dietrich Burde,
Faculty of Mathematics\\
University of Vienna\\
Oskar-Morgenstern-Platz 1\\
1090 Vienna\\
Austria}
\email{dietrich.burde@univie.ac.at}

\author[Globke]{Wolfgang Globke}

\address{Wolfgang Globke,
School of Mathematical Sciences\\
The University of Adelaide\\
SA 5005\\
Australia}
\email{wolfgang.globke@adelaide.edu.au}


\begin{abstract}
A complex vector space $V$ is a prehomogeneous $G$-module if
$G$ acts rationally on $V$ with a Zariski-open orbit. 
The module is called \'etale if $\dim V=\dim G$.
We study \'etale modules for reductive algebraic groups $G$
with one-dimensional center.
For such $G$, even though every \'etale module is a
regular prehomogeneous module, its irreducible submodules have
to be non-regular.
For these non-regular prehomogeneous modules, we determine some
strong constraints on the ranks of their simple factors.
This allows us to show that there do not exist \'etale modules for
$G=\GL_1\times S\times\cdots\times S$, with $S$ simple.
\end{abstract}

\setcounter{tocdepth}{1}

\maketitle

%
%

\tableofcontents

\input{inputIntro2}

\input{inputPV}

\input{inputCastling}

\input{inputRegular}

\input{inputTables}
\input{inputEtale}


\input{inputBibliography}

\end{document}

%% file: inputIntro2.tex

\section{Introduction}\label{sec_intro}
Affine \'etale representations of Lie groups arise in many contexts. 
For a given connected Lie group $G$, the existence of such a representation is equivalent 
to the existence of a left-invariant affine structure on $G$
(see \cite{baues,burde1}).
In $1977$ Milnor \cite{MIL} discussed
the importance of such structures for the study of fundamental groups of complete affine manifolds,
and for the study of affine crystallographic groups, which initiated generalizations
of the Bieberbach theorems for Euclidean crystallographic groups to affine crystallographic groups, 
see \cite{GRS}.
Milnor asked the existence question for left-invariant affine
structures on a given Lie group $G$, and 
suggested that all solvable Lie groups $G$ admit such a structure.
This question received a lot of attention, and was eventually
answered negatively by Benoist \cite{benoist}.
For a survey on the results and the history see
\cite{BU24,BU51,GRS}.

Affine \'etale representations of $G$ and left-invariant affine structures on $G$ both define a bilinear 
product on the Lie algebra $\frg$ of $G$ that gives $\frg$ the structure of a
\emph{left-symmetric algebra} (\emph{LSA-structure} for short),
and conversely an LSA-structure determines an affine structure
on $G$ (see Paragraph \ref{subsec:etaleLSA} below).
The existence question then can be formulated
on the Lie algebra level, and has been studied for several classes of Lie algebras, e.g., for semisimple, reductive, 
nilpotent and solvable Lie algebras, see \cite{BU24}. LSA-structures on Lie algebras also correspond to non-degenerate
involutive set-theoretical solutions of the Yang-Baxter equation, and to certain left brace structures, 
see \cite{ESS, BAC}.
A natural generalization of LSA-structures is given by post-Lie algebra structures 
on pairs of Lie algebras \cite{BU51}.

\'Etale representations also appear in the classification of
adjoint orbits on graded semisimple Lie algebras
$\frg=\bigoplus_{k\in\ZZ}\frg_k$.
The classification of $G_0$-orbits of nilpotent elements
can be reduced to determining certain graded semisimple
subalgebras $\frs$ associated to such elements which contain an
\'etale representation for the grade-preserving subalgebra
$\frs_0$ on the module $\frs_1$, see \cite{vinberg}.

It follows from the Whitehead Lemma in Lie algebra cohomology that a semi\-simple Lie algebra over a field $K$
of characteristic zero does not admit an LSA-structure.
The reductive Lie algebra $\ggl_n(K)$, however, admits a canonical LSA-structure.
Indeed, it is natural to consider the \emph{reductive} case, where
we have the powerful tools of invariant theory and representation theory for reductive groups at hand.
Furthermore, we can use the theory of prehomogeneous modules for reductive groups as developed by 
Sato and Kimura.
Still, it turns out that the existence question is already very difficult in the reductive
case, and is still open in general.

On the other hand there are several results for reductive Lie algebras -- respectively reductive groups -- with 
one-dimensional center.
Like the present work, these results make use of the invariant
theoretic methods to study rational modules for semisimple groups
with orbits of codimension $1$ introduced by Baues
\cite[Section 3]{baues}.
The first author showed in \cite[Theorem 2]{burde2} that a reductive Lie algebra 
$\frg=\fra\oplus\frs$ with $\frs$ split simple and $\dim \mathrm{Z}(\frg)=1$ admits an LSA-structure if and only if 
$\frs= \ssl_n(K)$.
Baues \cite[Section 5]{baues} classified all LSA-structures on $\ggl_n(K)$.

It is the aim of this article to make further progress for the reductive case with one-dimensional center.

\subsection{Reductive prehomogeneous modules}
A \emph{pre\-ho\-mo\-ge\-neous module} $(G,$ $\rho,V)$ consists of a linear algebraic group $G$ 
and a rational representation $\rho : G\to\GL(V)$ on a finite-dimensional complex vector
space $V$, such that $G$ has a Zariski-open orbit in $V$. The vector space $V$ is called a
\emph{prehomogeneous vector space}. We always assume that the representation $\rho$ is faithful
up to a finite subgroup. From now on $G$ is assumed to be reductive.
Recall that in this case, the Lie algebra $\frg$ of $G$ is a direct sum $\frg=\fra\oplus\frs$, where $\fra$ 
is the center of $\frg$, and $\frs$ is semisimple. We will call a prehomogeneous module $(G,\rho,V)$ for a reductive
group $G$ a \emph{reductive pre\-ho\-mo\-ge\-neous module}. A reductive group $G$ is called \emph{$k$-simple} 
if its semisimple factor has $k$ simple factors.

There are several classification results on reductive prehomogeneous modules by a group of Japanese mathematicians
around Mikio Sato and Tatsuo Kimura from the 1970s up to the present. However, a complete classification of 
prehomogeneous modules is not available.

The first classification result on prehomogeneous modules is due to Sato and Kimura \cite{SK}. They classified
\emph{irreducible} and \emph{reduced} prehomogeneous modules for reductive algebraic groups (the terminology 
will be explained in Section \ref{sec_PV}).
In addition, they determined the stabilizer subgroups of the
open orbits and the relative invariants for all cases.
We will label each module in this class by SK $n$, where $n$ is its number in \cite[\S 7]{SK}.
This classification can also be found in Kimura's book \cite{kimura}.

Kimura \cite[\S 3]{kimuraS} classified prehomogeneous modules of one-simple reductive groups,
\[
(\GL_1^k\times S,\rho_1\oplus\ldots\oplus\rho_k,V_1\oplus\ldots\oplus V_k)
\]
where $S$ is a simple group. We will label them Ks $n$, where $n$ is the number of the module in \cite[\S 3]{kimuraS}.
In each case, the generic isotropy subgroup is determined. This classification included non-irreducible modules.

Furthermore, Kimura et al.~\cite[\S 3]{kimuraI}, \cite[\S 5]{kimuraII} studied the prehomogeneity of modules 
for two-simple groups,
\[
(\GL_1^k\times S_1\times S_2,
\rho_1\oplus\ldots\oplus\rho_k,V_1\oplus\ldots\oplus V_k),
\]
where $S_1$ and $S_2$ are simple groups, under the assumption that one independent scalar multiplication
acts on each irreducible component. This assumption is a non-trivial simplification of the problem,
especially for the modules studied in \cite{kimuraII}, as it is far from obvious if one of these modules 
could be prehomogeneous with less than $k$ factors $\GL_1$ acting on the module. They studied two types 
of two-simple modules, I and II, and we will label them KI $n$ and KII $n$, where $n$ is their number
in \cite{kimuraI} or \cite{kimuraII}, respectively.

\subsection{\'Etale representations and LSA-structures}\label{subsec:etaleLSA}
It is clear that $\dim G\geq \dim V$ holds for any prehomogeneous module. If equality holds, $\dim G=\dim V$, 
we say that the representation $\rho$ (the module $V$) is an \emph{\'etale
representation} (an \emph{\'etale module}). More generally, one considers \emph{affine} \'etale
representations for arbitrary algebraic groups.
For reductive algebraic groups they can always be assumed to be
linear.

The existence of \'etale representations for a reductive algebraic group $G$ implies the existence of
LSA-structures on the reductive Lie algebra $\frg$ of $G$. More precisely, if $\rho'=(\d\rho)_1$ denotes the 
induced representation of $\rho':\frg\to\ggl(V)$ (also called an \'etale representation)
and $v\in V$ is a point in the open orbit of $\rho(G)$, then
\[
x\cdot y = \mathrm{ev}_v^{-1}(\rho'(x)\mathrm{ev}_v(y)),
\quad x,y\in\frg
\]
defines an LSA-structure on $\frg$. Here, $\mathrm{ev}_v:\frg\to V$ denotes the evaluation map
$x\mapsto\rho'(x)v$ at the point $v$. It is invertible since $\dim\frg=\dim V$ for an \'etale representation.
In addition the LSA-structure determines a left-invariant flat torsion-free affine connection $\nabla$ on $G$, 
by setting
\[
\nabla_x y = x\cdot y.
\]

\subsection{Overview and results}

The aim is, as said, to make progress on the structure of \'etale modules for reductive algebraic groups with
one-dimensional center. We briefly recall the theory of prehomogeneous modules as developed
by Sato and Kimura \cite{SK} in Section \ref{sec_PV}. We study some combinatorial aspects of castling 
transforms of irreducible reductive prehomogeneous modules in Section \ref{sec:castling}.
We find a rather strong constraint on which groups can appear as castling transforms:

\begin{mthm}\label{mthm:castling}
Let $(G,\rho,V)$ be an irreducible
prehomogeneous module for a reductive algebraic group.
Then:
\[
(G,\rho,V)
=
(L\times\SL_{m_1}\times\cdots\times\SL_{m_k},\
\sigma\otimes\omega_1\otimes\cdots\otimes\omega_1,\
V^n\otimes\CC^{m_1}\otimes\cdots\otimes\CC^{m_k}),
\]
where $L$ is a reductive algebraic group with one simple factor,
$\sigma$ is irreducible,
and $n,m_1,\ldots,m_k\geq 1$ such that
\begin{enumerate}
\item
$\gcd(m_i,m_j)=1$ for $1\leq i<j\leq k$.
\item
$\gcd(n,m_i)=1$ for all but at most one
index $i_0\in\{1,\ldots,k\}$.
\end{enumerate}
Moreover, if $(G,\rho,V)$ is castling-equivalent to a
one-simple irreducible module, then part (2) holds for all $i\in\{1,\ldots,k\}$.
\end{mthm}

Some general properties of \'etale representations (not just for
reductive groups) are reviewed in Section \ref{sec:general}.
We show that every reductive \'etale module is regular,
and that unipotent and semisimple algebraic groups do not admit
(linear) \'etale representations.
In Section \ref{sec_special_examples} we identify the \'etale
modules among certain classifications of prehomogeneous modules
due to Sato and Kimura \cite{SK} and Kimura et al.~\cite{kimuraS,kimuraI}.
In Section \ref{sec_etale} we derive
criteria for reductive algebraic groups $G$ with one-dimensional
center to admit \'etale representations.
By Lemma \ref{lemma_regspecial}, an \'etale module for $G$
is a regular prehomogeneous module (in the sense of Section
\ref{sec_PV}),
but this is not necessarily true for the submodules of an
\'etale module.
A main tool for the investigation of reducible
modules is the following theorem proved by Baues \cite[Lemma 3.7]{baues}.

\begin{bthm}[Baues]
Let $G=\GL_1\times S$ with $S$ semisimple and let $(G,\rho,V)$ be
an \'etale module.
Suppose $(G,\rho,W)$ is a proper submodule of $(G,\rho,V)$.
Then $(G,\rho,W)$ is a non-regular prehomogeneous module.
\end{bthm}

Combining this with Theorem \ref{mthm:castling}, we find a
non-existence result for a certain class of reductive groups:

\begin{mthm}\label{mthm:noetale_SLn}
Let $G=\GL_1\times S\times\overset{k}{\cdots}\times S$,
where $S$ is a simple algebraic group and $k\geq 2$.
Then $G$ has no \'etale representations.
\end{mthm}

\subsection*{Notations and conventions}

We write $V^m$ to emphasize that the dimension of a vector space
$V$ is $m$.

The unit element of a group $G$ is denoted by $1$ or $1_G$.
For matrix groups, we also use $I_n$ to denote the identity matrix.
When writing $\GL_n$ (resp. $\SL_n$, $\Sp_n$, $\SO_n$, $\Spin_n$),
we always assume the complex numbers as the coefficient field.

For convenience, we will often denote a module $(\rho,V)$
by the representation $\rho$ only.
In this case, we also write $\dim\rho$ for $\dim V$.
The dual representation (or module) is denoted by $(\rho^*,V^*)$.
The notation $\rho^{(*)}$ means either $\rho$ or its dual $\rho^*$.

It is well-known that an irreducible representation $\rho$
of a semisimple Lie algebra $\frg$ is uniquely determined by its
highest weight $\omega$.
After the choice of a Cartan subalgebra of $\frg$,
$\omega$ is a unique integral linear combination
$m_1\omega_1+\ldots+m_n \omega_n$ of the fundamental weights
$\omega_1,\ldots,\omega_n$ of $\frg$.
For brevity we often write $\omega$ when we mean the
``representation $\rho$ with highest weight $\omega$''.
The representation of $\ggl_1$ (or $\GL_1$)
by scalar multiplication on a
vector space is denoted by $\mu$.
The trivial representation for any group is denoted by $1$.

\subsection*{Acknowledgements}
Dietrich Burde acknowledges support by the Austrian Science Foundation FWF, grant P28079 and grant I3248.
Wolfgang Globke is supported by the Australian Research Council grant DE150101647.
He would also like to thank Oliver Baues for introducing him to this subject several years ago, and for many clarifying comments.
We would also like to thank an anonymous referee for
suggesting some clarifications in the text and pointing out
the references \cite{vinberg}.

%% file: inputPV.tex

\section{Basics of prehomogeneous modules}\label{sec_PV}

\subsection{Prehomogeneous modules and relative invariants}

Let $(G,\rho,V)$ be a prehomogeneous module.
The points $v$ in the open orbit of $G$ are called \emph{generic
points}, and the stabilizer $G_v=\{g\in G\mid gv=v\}$
at a generic point $v$ is
called the \emph{generic isotropy subgroup}, its Lie algebra
$\frg_v$ is called the \emph{generic isotropy subalgebra}.
The \emph{singular set} $V_0=V\backslash\rho(G)v$
is the complement of the open orbit in $V$.

Prehomogeneity is equivalent to
\[
\dim G_v=\dim G - \dim V
\]
and to $V=\{\d\rho(A)v\mid A\in\frg\}$.
In particular, if $\rho$ is \'etale, then $G_v$ is a finite
(since algebraic) subgroup and $\d\rho(\cdot)v:\frg\to V$ is a vector
space isomorphism.

Prehomogeneous modules are to a large extent characterized
by their \emph{relative invariants}, that is, those rational
functions $f:V\to\CC$ satisfying
\[
f(gv) = \chi(g) f(v),
\]
where $g\in G$ and $\chi\in\Ch(G)=\{\chi:G\to\CC^\times\mid \chi
\text{ is a rational homomorphism}\}$.
Prehomogeneity of $(G,\rho,V)$ is equivalent to the fact that
any absolute invariant (that is, with character $\chi=1$)
is a constant function.

Given a relative invariant $f$ of $(G,\rho,V)$, define a map
\[
\phi_f:V\backslash V_0\to V^*,
\quad x\mapsto\grad\log f(x).
\]
If the image of $\phi_f$ is Zariski-dense in $V^*$, then we
call $f$ a \emph{non-degenerate} relative invariant, and
$(G,\rho,V)$ a \emph{regular} prehomogeneous module.
For reductive algebraic groups $G$, we have the following
characterization of regular prehomogeneous modules
(see Kimura \cite[Theorem 2.28]{kimura}).

\begin{thm}\label{thm_reductive}
Let $G$ be a reductive algebraic group and $(G,\rho,V)$ a
prehomogeneous module.
Then the following are equivalent:
\begin{enumerate}
\item
$(G,\rho,V)$ is a regular prehomogeneous module.
\item
The singular set $V_0$ is a hypersurface.
\item
The open orbit $\rho(G)v=V\backslash V_0$ is an affine
variety.
\item
Each generic isotropy subgroup $G_v$ for $v\in V\backslash V_0$
is reductive.
\item
Each generic isotropy subalgebra $\frg_v$ for $v\in V\backslash V_0$
is reductive in $\frg=\Lie(G)$.
\end{enumerate}
\end{thm}

\subsection{Castling and promotion}

Two modules $(G_1,\rho_1,V_1)$ and
$(G_2,\rho_2,V_2)$ (or representations $\rho_1$ and $\rho_2$)
are called \emph{equivalent} if there exists
an isomorphism of groups $\psi:\rho_1(G_1)\to\rho_2(G_2)$ and a
linear isomorphism $\phi:V_1\to V_2$ such that
$\psi(\rho_1(g))\phi(x)=\phi(\rho_1(g)x)$ for all $x\in V_1$
and $g\in G_1$.

\begin{remark}\label{rem:dual}
If $G$ is reductive, then the dual representation
$\rho^*:G\to\GL(V^*)$ of any given representation $\rho:G\to\GL(V)$
is equivalent to $\rho$.
This follows from a result by Mostow \cite{mostow}.
\end{remark}

Let $m>n\geq 1$ and $\rho:G\to\GL(V^m)$ be a finite-dimensional
rational representation of an algebraic group $G$.
Then we say the modules
\[
\bigl( G\times \GL_n,\ \rho\otimes\omega_1,\ V^m \otimes \CC^n \bigr)
\quad \mathrm{and}\quad
\bigl( G\times \GL_{m-n},\ \rho^*\otimes\omega_1,\ V^{m*} \otimes \CC^{m-n} \bigr)
\]
are \emph{castling transforms}\index{castling transform} of each
other.
More generally, we say two modules $(G_1,\rho_1,V_1)$ and
$(G_2,\rho_2,V_2)$ are \emph{castling-equivalent}
if $(G_1,\rho_1,V_1)$ is equivalent to a module
obtained after a finite number of castling transforms from
$(G_2,\rho_2,V_2)$.

We say a module $(G,\rho,V)$ is \emph{reduced} (or \emph{castling-reduced}) if $\dim V\leq \dim V'$ for every castling transform
$(G,\rho',V')$ of $(G,\rho,V)$.

\begin{thm}[Sato \& Kimura \cite{SK}]\label{thm_sato}
Let $m>n\geq 1$ and $\rho:G\to\GL(V^m)$ be a finite-dimensional
rational representation of an algebraic group $G$.
Then 
\[
\bigl( G\times \GL_n,\ \rho\otimes\omega_1,\ V^m \otimes \CC^n \bigr)
\]
is a prehomogeneous module (with generic isotropy subgroup $H^{(n)}$)
if and only if its castling transform
\[
\bigl( G\times \GL_{m-n},\ \rho^*\otimes\omega_1,\ V^{m*} \otimes \CC^{m-n} \bigr)
\]
is prehomogeneous (with generic isotropy subgroup $H^{(m-n)}$).
Furthermore, $H^{(n)}$ and $H^{(m-n)}$ are isomorphic.

Addendum: If $G$ is reductive and its center acts by scalar multiplication
on $V^m\otimes\CC^n$, then
we can replace every occurrence of $\GL_n$ by $\SL_n$ in the
above statement, and prehomogeneity of
 \[
\bigl( G\times \SL_n,\ (\sigma\otimes 1)\oplus(\rho\otimes\omega_1),\ V^k\oplus (V^m \otimes \CC^n) \bigr)
\]
is equivalent to prehomogeneity of
\[
\bigl( G\times \SL_{m-n},\ (\sigma^*\otimes 1)\oplus(\rho\otimes\omega_1),\ V^{k*}\oplus (V^m \otimes \CC^{m-n}) \bigr).
\]
\end{thm}

\begin{remark}
Castling transforms regular prehomogeneous modules into regular
prehomogeneous modules, and
\'etale modules into \'etale modules (because $H^{(n)}\cong H^{(m-n)}$).
\end{remark}

\begin{example}\label{ex:promotion}
Castling allows to add additional factors to the group.
Let $(G,\rho,V^m)$ be a reductive prehomogeneous module.
We can interpret it as
\[
(G,\rho,V^m)=(G\times\SL_1,\rho\otimes\omega_1,V^m\otimes\CC^1).
\]
The castling transform of this module is
\[
(G\times\SL_{m-1},\rho\otimes\omega_1,V^m\otimes\CC^{m-1}).
\]
We call a castling transform of this particular type
a \emph{promotion} of the module $(G,\rho,V^m)$.
\end{example}

%% file: inputCastling.tex
%
%
\section{Castling transforms of irreducible prehomogeneous modules}\label{sec:castling}

In this section we do not assume that the
prehomogeneous modules are \'etale modules.
Our aim is to prove Theorem \ref{mthm:castling}.

\begin{lem}\label{lem:gcd}
Let $L$ be a reductive algebraic group with one simple factor and let
\[
(L\times\SL_m,\ \sigma\otimes\omega_1,\ V^n\otimes\CC^m)
\]
be a module with $(L,\sigma)\neq(\SL_n,\omega_1)$,
$\sigma$ irreducible and $m, n\geq 1$.
Let
\[
(L\times\SL_{m_1}\times\cdots\times\SL_{m_k},\
\sigma\otimes\omega_1\otimes\cdots\otimes\omega_1,\
V^n\otimes\CC^{m_1}\otimes\cdots\otimes\CC^{m_k})
\]
be castling-equivalent to the first module, with $k\geq 2$. Then:
\begin{enumerate}
\item
$\gcd(m_i,m_j)=1$ for $1\leq i<j\leq k$.
\item
$\gcd(n,m_i)=1$ for all but at most one
index $i_0\in\{1,\ldots,k\}$.
\item
$\gcd(n,m_{i_0})=\gcd(n,m)$ for this index $i_0$.
\end{enumerate}
\end{lem}
\begin{proof}
Any sequence of castling transforms of the original module will
start with a promotion, adding a factor
$(\SL_{nm-1},\omega_1,\CC^{nm-1})$
to the module. But clearly, $\gcd(m,nm-1)=1=\gcd(n,nm-1)$,
and for $m_1=m$, $\gcd(n,m_1)=\gcd(n,m)$.

Suppose the claim holds after $\ell\geq 1$ castling transforms of 
the original module.
We may assume the groups are ordered such that $i_0=1$.
Apply another castling transform.
If the transform is a promotion, then we obtain a new factor
$\SL_{nm_1\cdots m_k-1}$, and the claim clearly holds for the new
module.
Otherwise, consider two cases:

First, suppose that $\SL_{m_1}$ is replaced
by $\SL_{m_1'}$ with $m_1'=nm_2\cdots m_k-m_1$.
By the induction hypothesis, for $i=2,\ldots,k$, we have
$\gcd(n,m_i)=1$ and also $\gcd(m_1',m_i)=1$, since every divisor of
$m_i$ divides $nm_2\cdots m_k$ but not $m_1$.
Suppose $d$ is a common divisor of $n$ and $m_1$. Then $d$ divides
both $nm_2\cdots m_k$ and $m_1$, hence $d$ divides $m_1'$.
Similarly, every divisor of $n$ and $m_1'$ divides $m_1$.
Hence $\gcd(n,m_1)=\gcd(n,m_1')$.

Now consider the case that a factor other than $\SL_{m_1}$,
say $\SL_{m_k}$, is replaced by the castling transform.
The new factor is $\SL_{m_k'}$ with $m_k'=nm_1\cdots m_{k-1}-m_k$.
By the induction hypothesis, every divisor of $m_i$ with $i\neq k$
divides $nm_1\cdots m_{k-1}$ but not $m_k$, hence not $m_k'$.
It follows that $\gcd(m_i,m_k')=1$ for all $i\neq k$.
Similarly, $\gcd(n,m_k')=1$.
Moreover, $\gcd(n,m_i)=1$ for all $2\leq i<k$ and
$\gcd(n,m_1)=\gcd(n,m)$ by the induction hypothesis.

So the claim holds after $\ell+1$ castling transforms, and the
lemma follows.
\end{proof}

\begin{lem}\label{lem:gcd_SLn}
For $m,n\geq 1$, let
\[
(L\times\GL_m,\ \sigma\otimes\omega_1,\ V^n\otimes\CC^m)
\]
be a module with $(L,\sigma)=(\SL_n,\omega_1)$ or
$(L,\sigma)=(\SL_3\times\SL_3,\omega_1\otimes\omega_1)$.
In the latter case, assume additionally that $\gcd(3,m)=1$.
Let
\[
(\SL_{m_1}\times\cdots\times\SL_{m_k},\
\omega_1\otimes\cdots\otimes\omega_1,\
\CC^{m_1}\otimes\cdots\otimes\CC^{m_k})
\]
be castling-equivalent to the first module, with $k\geq 2$. Then
$\gcd(m_i,m_j)=1$ for all $1\leq i<j\leq k$
but at most one pair of indices $i_0,j_0\in\{1,\ldots,k\}$.
\end{lem}
The proof is mutatis mutandis identical to the proof of Lemma
\ref{lem:gcd}.

{
\renewcommand{\themthm}{\ref{mthm:castling}}
\begin{mthm}
Let $(G,\rho,V)$ be an irreducible
prehomogeneous module for a reductive algebraic group.
Then:
\[
(G,\rho,V)
=
(L\times\SL_{m_1}\times\cdots\times\SL_{m_k},\
\sigma\otimes\omega_1\otimes\cdots\otimes\omega_1,\
V^n\otimes\CC^{m_1}\otimes\cdots\otimes\CC^{m_k}),
\]
where $L$ is a reductive algebraic group with one simple factor,
$\sigma$ is irreducible,
and $n,m_1,\ldots,m_k\geq 1$ such that
\begin{enumerate}
\item
$\gcd(m_i,m_j)=1$ for $1\leq i<j\leq k$.
\item
$\gcd(n,m_i)=1$ for all but at most one
index $i_0\in\{1,\ldots,k\}$.
\end{enumerate}
Moreover, if $(G,\rho,V)$ is castling-equivalent to a
one-simple irreducible module, then part (2) holds for all $i\in\{1,\ldots,k\}$.
\end{mthm}
\addtocounter{mthm}{-1}
}
\begin{proof}
Every irreducible reductive prehomogeneous module is castling-equivalent to one of those classified by Sato and Kimura \cite[\S 7]{SK}, so it is enough to prove the theorem for those modules.
Lemma \ref{lem:gcd_SLn} proves the theorem for the irreducible
modules SK I-1 (with $(\SL_m,\omega_1)$), SK I-12, SK III-1
and SK III-2.
From the classification it is clear that every other reduced
irreducible module is of the form assumed in Lemma \ref{lem:gcd},
and so the theorem follows from this lemma in these cases.
\end{proof}

\begin{remark}
Every reductive prehomogeneous module decomposes into irreducible
ones, but since such a decomposition can be obtained by taking direct sums (that is, $(G_1,\rho_1,V_1)\oplus(G_2,\rho_2,V_2)=(G_1\times G_2,\rho_1\oplus\rho_2,V_1\oplus V_2))$,
it is not true that every
castling transform of non-irreducible prehomogeneous modules is
of the form described in Theorem \ref{mthm:castling}.
\end{remark}

%% file: inputRegular.tex
%
%
\section{General properties of \'etale representations}\label{sec:general}

\'Etale modules were introduced in Section \ref{sec_intro}.
Here we present some structural results on \'etale modules,
whereas in the next section we provide many new examples of
\'etale modules for reductive groups.

\begin{prop}\label{cor_iso}\label{prop_iso}
The following conditions are equivalent:
\begin{enumerate}
\item $(G,\rho_1\oplus\rho_2,V_1\oplus V_2)$ is an \'etale
module.
\item $(G,\rho_1,V_1)$ is prehomogeneous and $(H,\rho_2|_H,V_2)$
is an \'etale module, where $H$ denotes the connected
component of the generic isotropy subgroup of $(G,\rho_1,V_1)$.
\end{enumerate}
Equivalence also holds if each ``\'etale'' is replaced by
``prehomogeneous''.
\end{prop}
\begin{proof}
By Kimura \cite[Lemma 7.2]{kimura}, $\rho_1\oplus\rho_2$ is
prehomogeneous if and only if $\rho_1$ and $\rho_2|_H$ are.
The representation $\rho_1\oplus\rho_2$ is \'etale if and only if
it is prehomogeneous and $\dim G=\dim V_1+\dim V_2$.

Suppose $\rho_1\oplus\rho_2$ is \'etale.
So $(G,\rho_1,V_1)$ and
$(H,\rho_2|_H,V_2)$ are prehomogeneous, and since
$\dim G-\dim H=\dim V_1$,
we have
$\dim H=\dim V_2$,
so $(H,\rho_2|_H,V_2)$ is \'etale.

Conversely, if we assume $(G,\rho_1,V_1)$ to be prehomogeneous and
$(H,\rho_2|_H,V_2)$ to be \'etale,
then $(G,\rho_1\oplus\rho_2,V_1\oplus V_2)$ is obviously prehomogeneous
and as 
\[
\dim G-\dim V_1=\dim H=\dim V_2,
\]
it is even \'etale.
\end{proof}

\begin{cor}\label{cor_redequ}
Let $(G,\rho_1,V_1)$ be a prehomogeneous module with reductive
generic isotropy subgroup. Then $(G,\rho_1\oplus\rho_2,V_1\oplus V_2)$ is \'etale if and only
if $(G,\rho_1\oplus\rho_2^*,V_1\oplus V_2^*)$ is \'etale,
and then their generic isotropy subgroups are isomorphic.
\end{cor}
\begin{proof}
For a reductive group $G$, $(G,\rho,V)$ is equivalent to
$(G,\rho^*,V^*)$, see Remark \ref{rem:dual}.
The corollary now follows from Proposition \ref{cor_iso}.
\end{proof}

\subsection{Regularity of \'etale modules}

First, we note that non-regular prehomogeneous modules are not
\'etale modules for a reductive algebraic group.

\begin{lem}\label{lemma_regspecial}
Let $G$ be a reductive algebraic group.
If $(G,\varrho,V)$ is an \'etale module,
then it is a regular pre\-homogeneous module.
\end{lem}
\begin{proof} The generic isotropy subgroup of an \'etale module is
finite, hence reductive.
By Theorem \ref{thm_reductive}, the module is regular.
\end{proof}

This lemma does not imply that any irreducible component of an
\'etale module must be regular.
In fact, it follows from Theorem \ref{mthm_nonreg}
that for groups with one-dimensional center, an \'etale module 
that contains a regular irreducible component
must be irreducible itself.

\subsection{Groups with trivial character group}
Let $\Ch(G)$ denote the character group of $G$
(the group of rational homomorphisms $\chi:G\to\CC^\times$).
The following proposition is possibly well-known. Since we do
not know a reference for it, we will give a proof here.

\begin{prop}\label{prop_trivchar}
Let $G$ be an algebraic group with $\Ch(G)=\{1\}$.
Then $G$ does not admit a rational linear \'etale representation.
\end{prop}
\begin{proof}
Assume that $\varrho:G\to V$ is a linear \'etale representation.
Let $n=\dim G=\dim V>0$.
By Kimura \cite[Proposition 2.20]{kimura},
the prehomogeneous module $(G,\varrho,V)$ has a relative
invariant $f$ of degree $n$, so $f$ is not constant.
As $\Ch(G)=\{1\}$, the associated character $\chi$ of $f$ must be
$\chi=1$, which means that $f$ is an absolute invariant.
But this is a contradiction to the fact that prehomogeneous
modules do not admit non-constant absolute invariants.
\end{proof}

We conclude that unipotent groups and semisimple groups do not
admit linear \'etale  representations, since their respective groups
of rational characters are trivial.

\begin{cor}\label{cor_noetale_ss}
There is no rational linear \'etale representation for a semisimple algebraic
group.
\end{cor}

\begin{cor}\label{cor_noetale_uni}
There is no rational linear \'etale representation for a uni\-potent algebraic
group.
\end{cor}

On the other hand, a unipotent algebraic group may admit an affine \'etale representation.
This is not the case for a semisimple algebraic group $S$. 
It is already known that $S$ does not admit an affine \'etale 
representation, because of the correspondence to LSA-structures on the semisimple Lie algebra $\frs$ of $S$. 
Indeed, a semisimple Lie algebra $\frs$ of characteristic zero does not admit an LSA-structure, because
$H^1(\frs,M)=0$ for all finite-dimensional $\frs$-modules $M$ by the first Whitehead
Lemma on Lie algebra cohomology.
A proof using this fact is given by Medina \cite{medina},
see also Burde \cite{burde2}.
However, the argument with the character group here gives an 
independent proof.
Proofs using the existence of a fixed point were given
by Helmstetter \cite{helmstetter} or Baues \cite{baues}.

\begin{remark}
The vanishing of the Lie algebra cohomology $H^n(\frg,\frg)$ for all $n\ge 0$ with the adjoint
module $\frg$ alone is not enough to ensure that $G$ does not admit an affine \'etale representation.
For example, the linear algebraic group $G={\rm Aff}(V)$ for a vector space $V$ has a cohomologically rigid 
Lie algebra $\frg=\mathfrak{aff}(V)$, which satisfies $H^n(\frg,\frg)=0$ for all $n\ge 0$ as a consequence of Carles \cite[Lemma 2.2]{carles}.
But the coadjoint representation of $G$ is \'etale
(cf.~Rais \cite{rais}).
\end{remark}

%% file: inputTables.tex

\section{\'Etale modules for groups with one or two simple factors}\label{sec_special_examples}

As stated in the introduction,
certain classification results for \'etale modules are immediately
obtained from the classification of prehomogeneous modules.
These classifications have been collected in a convenient
reference in \cite{table}.

%

\begin{remark}\index{scalar multiplication}
In Kimura et al.~\cite{kimuraS}, \cite{kimuraI}, the prehomogeneous modules are always stated with
one scalar multiplication $\mu$ acting on each irreducible component,
that is, $(\GL_1^k\times G,\rho_1\oplus\ldots\oplus\rho_k)$, and in this
case we do not explicitly state the scalar multiplications,
as it is understood that each $\rho_i$ stands for $\mu\otimes\rho_i$.
But in some cases, we do not need an independent scalar multiplication
on each component to achieve prehomogeneity.
Consider for example the prehomogeneous module Ks A-2, $(\GL_1^n\times\SL_n,\omega_1^{\oplus n})$.
For $\omega_1^{\oplus n}$ we need only the operation of $\SL_n$
and one scalar multiplication $\GL_1$ acting on all components to
obtain a prehomogeneous module,
that is $(\GL_1\times\SL_n, \mu\otimes\omega_1^{\oplus n})$.
\end{remark}

Finding the \'etale modules of type SK, Ks and KI is rather
easy, as the generic isotropy subgroup is known in each case.
Thus we can just pick the modules with $G_v^\circ\cong\{1\}$ from
the known classification tables.
Finding \'etale modules in the class KII is significantly more
complicated and will be done in a forthcoming article.

\begin{prop}\label{prop_SKspecial}
The following irreducible reduced prehomogeneous modules are all
\'etale modules in the list SK:
\begin{itemize}
\item SK I-4: $(\GL_2,3\omega_1,\Sym^3\CC^2)$.
\item SK I-8: $(\SL_3\times\GL_2,2\omega_1\otimes\omega_1,\Sym^2\CC^3\otimes\CC^2)$.
\item SK I-11: $(\SL_5\times\GL_4,\omega_2\otimes\omega_1,\bigwedge^2\CC^5\otimes\CC^4)$.
\end{itemize}
\end{prop}

\begin{prop}\label{prop_Ksspecial}
The following non-irreducible one-simple prehomogeneous modules are
all \'etale modules in the list Ks:
\begin{itemize}
\item Ks A-1 for $n=2$: This is equivalent to Ks A-4 with $n=2$.
\item Ks A-2: $(\GL_1\times\SL_n,\mu\otimes\omega_1^{\oplus n},(\CC^n)^{\oplus n})$.
\item Ks A-3: $(\GL_1^{n+1}\times\SL_n,\omega_1^{\oplus n+1},(\CC^n)^{\oplus n+1})$.
\item Ks A-4: $(\GL_1^{n+1}\times\SL_n,\omega_1^{\oplus n}\oplus\omega_1^*,(\CC^n)^{\oplus n}\oplus\CC^{n*})$.
\item Ks A-11 for $n=2$: $(\GL_1^2\times\SL_2,2\omega_1\oplus\omega_1,\Sym^2\CC^2\otimes\CC^2)$.
\item Ks A-12 for $n=2$: Equivalent to Ks A-11 with $n=2$.
\item Ks A-20 for $n=1$: Equivalent to Ks A-2 with $n=2$.
\end{itemize}
\end{prop}

\begin{cor}\label{cor_onesimple_etale}
If $(\GL_1^k\times S,\rho,V)$ for $k\geq 1$ and a simple group $S$
is an \'etale module, then $S=\SL_n$ for some $n\geq 1$.
\end{cor}

\begin{prop}\label{prop_KIspecial}
The following two-simple prehomogeneous modules of type I are
all \'etale modules in the list KI:
\begin{itemize}
\item KI I-1: $(\GL_1^2\times\SL_4\times\SL_2,(\omega_2\otimes\omega_1)\oplus(\omega_1\otimes\omega_1),(\bigwedge^2\CC^4\otimes\CC^2)\oplus(\CC^4\otimes\CC^2))$.
\item KI I-2: $(\GL_1^2\times\SL_4\times\SL_2,(\omega_2\otimes\omega_1)\oplus(\omega_1\otimes1)\oplus(\omega_1\otimes1),(\bigwedge^2\CC^4\otimes\CC^2)\oplus\CC^4\oplus\CC^4)$.
\item KI I-6: $(\GL_1^3\times\SL_5\times\SL_2,(\omega_2\otimes\omega_1)\oplus(\omega_1^*\otimes 1)\oplus(\omega_1^{(*)}\otimes 1),(\bigwedge^2\CC^5\otimes\CC^2)\oplus\CC^{5*}\oplus\CC^{5(*)})$.
\item KI I-16: $(\GL_1^2\times\Sp_2\times\SL_3,(\omega_1\otimes\omega_1)\oplus(\omega_2\otimes 1)\oplus(1\otimes\omega_1^*),(\CC^4\otimes\CC^3)\oplus V^5\oplus\CC^3)$.
\item KI I-18: $(\GL_1^3\times\Sp_2\times\SL_2,(\omega_2\otimes\omega_1)\oplus(\omega_1\otimes 1)\oplus(1\otimes\omega_1),(V^5\otimes\CC^2)\oplus \CC^4\oplus\CC^2)$.
\item KI I-19: $(\GL_1^3\times\Sp_2\times\SL_4,(\omega_2\otimes\omega_1)\oplus(\omega_1\otimes 1)\oplus(1\otimes\omega_1^*),(V^5\otimes\CC^4)\oplus \CC^4\oplus\CC^4)$.
\end{itemize}
\end{prop}

%% file: inputEtale.tex
\section[\'Etale modules for one-dimensional center]{\'Etale modules for groups with one-dimensional center}\label{sec_etale}

In this section we study \'etale modules for algebraic
groups $G=\GL_1\times S$, where $S$ is semisimple.

\subsection{Non-regularity of submodules}\label{subsec_special_torus1}


An important property of \'etale modules for groups with
one-dimensional center is given by the following theorem
due to Baues \cite[Lemma 3.7]{baues}.

\begin{thm}[Baues]\label{mthm_nonreg}
Let $G=\GL_1\times S$ with $S$ semisimple and let $(G,\rho,V)$ be
an \'etale module.
Suppose $(G,\rho,W)$ is a proper submodule of $(G,\rho,V)$.
Then $(G,\rho,W)$ is a non-regular prehomogeneous module.
\end{thm}

In the terminology of Rubenthaler \cite{rubenthaler},
this theorem states that $(G,\rho,V)$ is \emph{quasi-irreducible}.

%
%
%
%

A simple example illustrates the statement of Theorem
\ref{mthm_nonreg}:

\begin{example}\label{ex_nonreg}
The module Ks A-2, $(\GL_1\times\SL_n,\mu\otimes\omega_1^{\oplus n},(\CC^n)^{\oplus n})$ is an \'etale module.
We identify $(\CC^n)^{\oplus n}=\mat_n$, and then a relative
invariant is given by the determinant of $n\times n$-matrices.
This module decomposes into $n$ irreducible and non-regular
summands of type SK III-2, $(\GL_1\times\SL_n,\mu\otimes\omega_1,\CC^n)$ corresponding to action by matrix-vector
multiplication on each column of matrices in $\mat_n$.
\end{example}

Theorem \ref{mthm_nonreg} does not hold if the center of $G$ has
dimension $\geq 2$:

\begin{example}\label{ex_nonreg_center_dim2}
Consider the module KI I-2,
\[
(\GL_1^2\times\SL_4\times\SL_2,(\omega_2\otimes\omega_1)\oplus(\omega_1\otimes1)\oplus(\omega_1\otimes1),({\bigwedge}^2\CC^4\otimes\CC^2)\oplus\CC^4\oplus\CC^4).
\]
The first irreducible component of this module, $\omega_2\otimes\omega_1$, corresponds to the regular irreducible module
SK I-15 with parameters $n=6$, $m=2$ (recall that over the complex
numbers, $\SO_6$ and $\SL_4$ are locally isomorphic).
\end{example}



\subsection{Groups with copies of one simple factor only}

We now want to study reductive groups of the form
\begin{equation}
G=\GL_1\times S\times\overset{k}{\cdots}\times S,
\label{eq:center1S}
\end{equation}
where $S$ is simple and $k\geq 2$
(for $k=1$ see Section \ref{sec_special_examples}).
If there are any irreducible \'etale modules for such a group,
each of them must be castling-equivalent to 
one of the modules in Proposition \ref{prop_SKspecial},
\begin{itemize}
\item
SK I-4: $(\GL_2\times\SL_1, 3\omega_1\otimes\omega_1, \Sym^3\CC^2)$.
\item
SK I-8: $(\SL_3\times\GL_2, 2\omega_1\otimes\omega_1, \Sym^2\CC^3\otimes\CC^2)$.
\item
SK I-11: $(\SL_5\times\GL_4, \omega_2\otimes\omega_1,\bigwedge^2\CC^5\otimes\CC^4)$.
\end{itemize}

\begin{remark}\label{rem:only_SL_irred}
As castling only adds additional factors $\SL_m$, this list shows
that no irreducible \'etale representation for \eqref{eq:center1S}
with $S\neq\SL_m$ can exist.
\end{remark}

By Theorem \ref{mthm_nonreg}, any reducible \'etale module
decomposes into irreducible components, each of which is a
non-regular prehomogeneous module for $G$.
Therefore, by the Sato-Kimura classification
\cite[III on p.~147]{SK} (or \cite[Section 1]{table}), each irreducible
component is castling-equivalent to one of the following:
\begin{itemize}
\item
SK III-1: $(L\times\GL_m,\rho\otimes\omega_1, V^n\otimes\CC^m)$,
where $\rho:L\to\GL(V^n)$ is an $n$-dimensional irreducible
representation of a semisimple algebraic group $L$ $(\neq\SL_n)$
with $m>n\geq 3$.
\item
SK III-2:
$(\SL_n\times\GL_m,\omega_1\otimes\omega_1, \CC^n\otimes\CC^m)$
for $\frac{1}{2}m\geq n\geq 1$.
\item
SK III-3:
$(\GL_{2n+1},\omega_2,\bigwedge^2\CC^{2n+1})$ for $n\geq 2$.
\item
SK III-4:
$(\GL_2\times\SL_{2n+1},\omega_1\otimes\omega_2,
\CC^2\otimes\bigwedge^2\CC^{2n+1})$ for $n\geq 2$.
\item
SK III-5:
$(\Sp_n\times\GL_{2m+1},\omega_1\otimes\omega_1,
\CC^{2n}\otimes\CC^{2m+1})$ for $n>2m+1\geq 1$.
\item
SK III-6:
$(\GL_1\times\Spin_{10},\mu\otimes\mathrm{halfspinrep},\CC\otimes V^{16})$.
\end{itemize}
Now it is obvious that any castling transform of one of these
modules will have a group which has at least one factor
$\SL_m$ with $m\geq 2$. So among these there is not even a
prehomogeneous module for a group \eqref{eq:center1S} with
$S\neq\SL_m$.
Combined with Remark \ref{rem:only_SL_irred}, we have:

\begin{cor}\label{cor:S_not_SLn}
Let $G=\GL_1\times S\times\cdots\times S$, where $S$ is a simple algebraic
group other than $\SL_m$ for any $m\geq 2$.
Then there exist no \'etale modules for $G$.
\end{cor}

\begin{lem}\label{lem:noetale_irrep_SLn}
$G=\GL_1\times\SL_m\times\overset{k}{\cdots}\times\SL_m$ has no irreducible \'etale representations.
\end{lem}
\begin{proof}
If $G$ has an irreducible \'etale module $(G,\rho,V)$, then
it is castling-equivalent to one of the modules
SK I-4, SK I-8 or SK I-11 above. But in each of these cases,
we have two factors $\SL_{m_1}$, $\SL_{m_2}$ with $m_1\neq m_2$.
Any non-trivial castling transform of these modules would have at
least three simple factors. By Theorem \ref{mthm:castling},
this means there are at least two simple factors
$\SL_{m_1}$, $\SL_{m_2}$ with $\gcd(m_1,m_2)=1$, which again
means $m_1\neq m_2$.
\end{proof}

{
\renewcommand{\themthm}{\ref{mthm:noetale_SLn}}
\begin{mthm}
Let $G=\GL_1\times S\times\overset{k}{\cdots}\times S$,
where $S$ is a simple algebraic group and $k\geq 2$.
Then $G$ has no \'etale representations.
\end{mthm}
\addtocounter{mthm}{-1}
}
\begin{proof}
Consider $G=\GL_1\times S_1\times\cdots\times S_k$.
As we are interested in the case where all simple factors are
identical,  by Corollary \ref{cor:S_not_SLn}, we only need to
consider the case where all $S_i=\SL_{m_i}$ for $m_i\geq 2$.

Let $(G,\rho,V)$ be an \'etale module.
First, observe that the \'etale representation has at least
one irreducible factor on which at least two of
the factors, say $S_1$ and $S_2$, act non-trivially.
In fact, otherwise $(G,\rho,V)$ would be a direct sum of \'etale
modules $(\GL_1\times S_i,\rho_i,V_i)$, which is regular
for $\GL_1\times S_i$ by Lemma \ref{lemma_regspecial}, hence for
$G$, as the stabilizer on $V_i$ is the product of the $S_j$ with
$j\neq i$.
But the center of $G$ is
one-dimensional, so by Theorem \ref{mthm_nonreg},
an $(G,\rho,V)$ does not have proper regular submodules.
This would imply $k=1$, contradicting our assumption that
$k\geq 2$.

Let $(G,\rho_1,V_1)$ be an irreducible factor on which at
least two simple factors of $G$ act non-trivially.
By Lemma \ref{lem:noetale_irrep_SLn}, there are no irreducible
\'etale representations of $G$ if all simple factors are identical,
so we may assume the \'etale representation is reducible, and
by Theorem \ref{mthm_nonreg}, $(G,\rho_1,V_1)$ must be a
non-regular irreducible prehomogeneous module for $G$.
After removing simple factors contained in the generic stabilizer
of $(\rho_1,V_1)$ from $G$, we can assume that $(G,\rho_1,V_1)$
is castling equivalent to one of the reduced irreducible modules
SK III-1 (with $(L,\rho)\neq(\SL_n,\omega_1)$), SK III-2,
SK III-3, and SK III-4.
In each of these cases the group is of the form
$\GL_1\times\SL_m\times\SL_n$ with $m\neq n$.
By Theorem \ref{mthm:castling}, any of its castling transforms
has at least two factors $\SL_{m_1}$ and $\SL_{m_2}$ with
$m_1,m_2>1$ and $\gcd(m_1,m_2)=1$. In particular, it is not
possible that all simple factors of the group $G$ are identical.
\end{proof}

\begin{remark}
If we admit a center $\GL_1^k$, then we trivially obtain
\'etale modules with semisimple part $\SL_n\times\cdots\times\SL_n$ by taking direct sums of \'etale modules for $\GL_1\times\SL_n$.
\end{remark}

%
%
%
%
%
%
%
%
%
%
